\documentclass[final,3p,times]{elsarticle}

\usepackage{algorithm}
\usepackage{algorithmic}
\usepackage{multirow}
\usepackage{xcolor}
\usepackage{hyperref}
\usepackage{url}
\usepackage{booktabs,threeparttable,caption2}
\usepackage{amsthm,amsmath}
\usepackage{float}
\usepackage{mathrsfs}
\usepackage{latexsym}
\usepackage{tabularx}
\newtheorem{thm}{Theorem}

\newtheorem{lem}{Lemma}

\newdefinition{rmk}{Remark}

\newcaptionstyle{left}{
\usecaptionmargin\captionfont
{\flushleft\bfseries\captionlabelfont\captionlabel\par}
\mbox{\onelinecaption{\captiontext}{\captiontext}}
}

\begin{document}

\begin{frontmatter}


\title{Exponential integrators for large-scale stiff matrix Riccati differential equations\tnoteref{label1}}
\tnotetext[label1]{The work was supported by the Natural Science Foundation of Jilin Province of China
(20180101224JC)}
\author{Dongping Li}
\ead{lidp@ccsfu.edu.cn}
%
\cortext[cor1]{Corresponding author.}
\address{$^a$ Department of Mathematics, Jilin University, Changchun 130012, PR China\\
$^b$Department of Mathematics, Changchun Normal University, Changchun 130032, PR China}

\begin{abstract}
Matrix Riccati differential equations arise in many different areas and are particular important within the field of control theory.
In this paper we consider numerical integration for large-scale systems of stiff matrix Riccati differential equations. We show how to apply exponential Rosenbrock-type integrators to get approximate solutions. Two typical exponential integration schemes are considered.
The implementation issues are addressed and some low-rank approximations are exploited based on high quality numerical algebra codes. Numerical comparisons demonstrate that the exponential integrators can obtain high accuracy and efficiency for solving large-scale systems of stiff matrix Riccati differential equations.
\end{abstract}

\end{frontmatter}
\section{Introduction}

In this paper we are concerned with numerical methods for large-scale systems of stiff matrix Riccati differential equations (MRDEs) of the following form
\begin{equation}\label{1.1}
\left\{
\begin{array}{l}
X'(t)=AX(t)+X(t)D+Q-X(t)GX(t)=:F(X(t)),\\
X(0)=X_0,
\end{array}
\right.
\end{equation}
where $A\in \mathbb{R}^{M\times M},~D\in \mathbb{R}^{N\times N},~Q\in \mathbb{R}^{M\times N},~G\in \mathbb{R}^{N\times M}$ are given matrices and $X(t)\in \mathbb{R}^{M\times N}$ is the unknown matrix-valued function.
MRDEs of this form occur in many important applications such as optimal control, $H_\infty$-control, filtering, boundary value problems for systems of ODEs and many others (see e.g. \cite{Abou,Ascher,Ichikawa,Jacobs}). In most control problems, the coefficient matrices $A$ and $D$ are obtained from the discretization of operators defined on infinite dimensional spaces, and the fast and slow modes exist. This means that the associated MRDEs will be fairly large and stiff.

An important special case of (\ref{1.1}) is the symmetric MRDEs
\begin{equation}\label{1.2}
\left\{
\begin{array}{l}
X'(t)=AX(t)+X(t)A^T+Q-X(t)GX(t),\\
X(0)=X_0,
\end{array}
\right.
\end{equation}
here, $Q=Q^T, G=G^T$ and $X_0=X_0^T.$ It is obvious that the solution of symmetric MRDEs is symmetric as $X(t)^T$ is also a solution. The symmetric MRDEs is possibly the most widely studied equations due to its importance in linear-quadratic optimal control problems. 
Another special mention should be paid to MRDEs (\ref{1.1}) with $G=0,$ yieiding the so-called matrix Sylvester differential equations (MSDEs). For a thorough description of these equations and some qualitative issues, we refer the reader to \cite{Abou,Fital,Jacobs,Juang,Reid1972} and the references appearing therein.

 Many numerical methods have been developed in the past for solving MRDEs. Perhaps the most natural numerical technique is to rewrite (\ref{1.1}) as an $MN$-vector ODEs based on Kronecker product, and then to use a standard numerical integrators such as Runge-Kutta or linear
multi-step solvers \cite{Butcher}. However, these approaches are not suitable for solution of large stiff MRDEs. They are generally computationally expensive and hard to exploit the structure inherited in some large practical problems.
For stiff MRDEs, some matrix-valued versions of implicit time integration
schemes, such as the BDF, Rosenbrock methods have been explored through a direct time discretization of (\ref{1.1}), see e.g. \cite{Benner01,Choi,Dieci}. Recently,
some other unconventional numerical methods have been also developed for MRDEs and related problems, including splitting methods
\cite{Mena,Stillfjord1,Stillfjord2} and projection methods \cite{Jbilou,Hached,Koskela}, etc.

The aim of this paper is to introduce exponential integrators for large-scale stiff problems of the forms \ref{1.1} and \ref{1.2}. In the past two decades exponential integrators have become a popular tool for solving large-scale stiff semi-linear systems of ODEs
\begin{equation}\label{1.4}
\left\{
\begin{array}{l}
y'=Ly+f(y),~y\in R^N,~L\in R^{N\times N},\\
y(0)=y_0.
\end{array}
\right.
\end{equation}
 A general derivation of exponential integrators is based on the variation-of-constants formula
\begin{equation}\label{1.5}
\begin{aligned}
y(t_n+h_n)=e^{h_nL}y(t_n)+\int_{0}^{h_n}e^{(h_n-s)L}f\big(y(t_n+s)\big)ds.
\end{aligned}
\end{equation}
By approximating the nonlinear terms
$f\big(y(t_n+s)\big)$ by an appropriate algebra polynomial, various type of exponential integrators can be exploited. Different approximations result in different types of exponential integrators of either multi-step type or Runge-Kutta type, see e.g. \cite{MH2,MH4,MH3,AK}. A main advantage of exponential integrators with stiff order conditions don't suffer from an order reduction even if the matrix $L$ is a discretization of
a unbounded linear operator. For a full overview of exponential integrators and associated software, we refer the readers to the reviews \cite{Hochbruck2010,BV2005} and references therein.
Although the matrix differential equations (\ref{1.1}) can be reformulated as the form (\ref{1.4}) and solved by an exponential integrator, this approach will generate very large $L$ and not be appropriate.

In the present paper we propose matrix-valued versions exponential integrators for stiff MRDEs (\ref{1.1}). The methods provides an efficient alternative to implicit integrators for computing solutions of MRDEs. For large-scale systems, in many applications it is often observed, both practical and theoretical, the solution has low numerical rank and can be approximated by products of low-rank matrices \cite{Lang,Stillfjord3}. To utilize such structure, we introduce how the low-rank implementation can be applied to exponential integrators. Thus we are able to save computational and memory storage requirements compared to the simple application of exponential integrators.

The remainder of the paper is organized as follows. In Section 2, we give some basic results and properties of MRDEs. In Section 3, the exponential Rosenbroc-type methods are introduced for the application to the MRDEs. Section 4 we show some issues of implementation and the low-rank approximations for both typical exponential integration schemes are exploited. Section 5 is devoted to some numerical examples and comparisons with splitting methods of similar orders. Finally, we draw some conclusions in Section 6.

\section{Preliminaries}
We start with recalling a general result on the solution of the MRDEs.
The following result shows that the MRDEs (\ref{1.1}) can be equivalently written in an integral form (see e.g. \cite{Kucera}).

\begin{thm}\label{th1}The exact solution of the MRDEs (\ref{1.1}) is given by
\begin{equation}\label{2.1}
X(t)=e^{tA}X_0e^{tD}+\int_{0}^{t}e^{(t-\tau)A}Qe^{(t-\tau)D}d\tau-\int_{0}^{t}e^{(t-\tau)A}X(\tau)GX(\tau)e^{(t-\tau)D}d\tau.
\end{equation}
\end{thm}
\begin{proof}
The proof can be done directly by differentiating both sides.
\end{proof}
Specifically, over the time interval $[t_n,~t_n+h_n],$ by using a change of variables $t=t_n+sh_n$ in (\ref{2.1}) to give
\begin{equation}\label{2.2}
X(t_{n}+h_n)=e^{h_nA}X(t_n)e^{h_nD}+h_n\int_{0}^{1}e^{(1-s)h_nA}Qe^{(1-s)h_nD}ds-h_n\int_{0}^{1}e^{(1-s)h_nA}X(t_{n}+sh_n)GX(t_n+sh_n)e^{(1-s)h_nD}ds.
\end{equation}
The formula (\ref{2.2}) also holds for time-varying coefficient matrices $Q=Q(t)$ and $G=G(t).$ To make constructing methods for MRDEs easier, we use
\begin{equation}\label{2.3}
\mathcal{S}(X):=AX+XD
\end{equation}
to denote the linear operator from the right-hand side of MRDEs (\ref{1.1}), which is called Sylvester operator.  The operator exponential satisfy the following relation (see \cite{Behr}):
\begin{equation}\label{2.4}
e^{\mathcal{S}}(X)=e^\mathcal{H}e^\mathcal{P}(X)=e^{A}Xe^{D},
\end{equation}
here $\mathcal{H}(X)=AX$ and $\mathcal{P}(X)=XD.$
Then, expression (\ref{2.2}) has the simplified form
\begin{equation}\label{2.5}
X(t_n+h_n)=e^{h_n\mathcal{S}}(X (t_n))+h_n\int_{0}^{1}e^{(1-s)h_n\mathcal{S}}(Q)ds-h_n\int_{0}^{1}e^{(1-s)h_n\mathcal{S}}(X(t_n+sh_n)GX(t_n+sh_n))ds.
\end{equation}
The first term from the right-hand side of (\ref{2.5}) involves operator exponential and represents the homogenous part of the solution, whereas the other two terms consist of integrals, again involving operator exponential. A natural idea to construct exponential integrators is to approximate the integrals on the right-hand side of (\ref{2.5}) by a quadrature formula, in which only the nonlinearity term $XGX$ are approximated but the operator exponential are treated exactly. In particular, for MSDEs, we have the following result.
\begin{lem}\label{lem1} Let $A\in \mathbb{R}^{M\times M},~D\in \mathbb{R}^{N\times N},$ and~$Q(t)\in \mathbb{R}^{M\times N}$ be a sufficiently differential matrix-value function, then the exact solution of the matrix differential equations
\begin{equation}\label{2.6}
X'(t)=AX(t)+X(t)D+Q(t),~~X(t_n)=X_n
\end{equation}
can be represented by the expansion
\begin{equation}\label{2.7}
X(t_n+h_n)=e^{h_n\mathcal{S}}(X(t_n))+\sum_{j=1}^{m}h_n^{j}\varphi_{j}(h_n\mathcal{S})(Q^{(j-1)}(t_n))ds+R_m(h_n),
\end{equation}
where
\begin{equation}\label{2.8}
R_m(h_n)=\frac{h_n^{m+1}}{(m-1)!}\int_{0}^{1}\int_{0}^{1}s^m(1-\theta)^{m-1}e^{(1-s)h_n\mathcal{S}}(Q^{(m)}(t_n+\theta sh_n))d\theta ds,
\end{equation}
\begin{equation}\label{2.9}
\varphi_j(z)=\int_0^1e^{(1-\theta)z}\frac{\theta^{j-1}}{(j-1)!}d\theta,~~~~  j\geq1.
\end{equation}
\end{lem}
\begin{proof}
By formula (\ref{2.4}), the solution of equations (\ref{2.6}) can be written
\begin{equation}\label{2.10}
X(t_n+h_n)=e^{h_n\mathcal{S}}X_n+h_n\int_{0}^{1}e^{(1-s)h_n\mathcal{S}}(Q(t_n+s h_n))ds.
\end{equation}
Inserting the Taylor series expansion of $Q(t_n+sh_n)$
\begin{equation}\label{2.11a}
Q(t_n+sh_n)=\sum\limits_{j=1}^m\frac{(sh_n)^{j-1}}{(j-1)!}Q^{(j-1)}(t_n)+\int_{0}^{1}\frac{(sh_n)^{m}}{(m-1)!}(1-\theta)^{m-1}Q^{(m)}(t_n+\theta sh_n))d\theta
\end{equation}
into the formula (\ref{2.10}) and applying the definition (\ref{2.9}) we arrive at the required result.
\end{proof}

The functions $\varphi_j(z)$ defined in (\ref{2.9}) satisfy the following recurrence relations
 \begin{equation}\label{2.11b}
\varphi_{j+1}(z)=\frac{\varphi_{j}(z)-\frac{1}{j!}}{z}, ~~~~\varphi_{0}(z)=e^z.
\end{equation}
A special case of the nonhomogeneous term in equations (\ref{2.6}) is an matrix polynomials, i.e., $Q(t)=\sum\limits_{j=0}^{m-1}\frac{t^{j-1}}{(j-1)!}N_j,$ $N_j\in \mathbb{R}^{M\times N},~~j=0,1,\cdots,m-1,$ in this case, the exact solution of equations (\ref{2.6}) can be represented by the expansion
\begin{equation}\label{2.13}
X(t)=e^{t\mathcal{S}}N_0+\sum\limits_{j=1}^{m}t^{j}\varphi_{j}(t\mathcal{S})N_j.
\end{equation}

\section{Exponential Rosenbrock-type integrators for MRDEs}

In this section, we consider the time discretization of MRDEs (\ref{1.1}). Rewrite equations (\ref{1.1}) as
\begin{equation}\label{3.1}
X'(t)=\mathcal{S}_n(X)+\mathcal{G}_n(X),
\end{equation}
where $\mathcal{S}_n$ denotes the Fr\'{e}chet derivative of $F$ and $\mathcal{G}_n$ the nonlinear remainder at $X_n,$ respectively:
\begin{eqnarray}\label{3.3}
\mathcal{S}_n(X)=A_nX+XD_n,~~\mathcal{G}_n(X)=F(X)-\mathcal{S}_n(X)
\end{eqnarray}
with $A_n=(A-X_nG)$ and $D_n=(D-GX_n).$

It is obvious that $\mathcal{S}_n$ is a Sylvester operator. Formally, by the variation of constants formula (\ref{2.5}), the exact solution of (\ref{3.1}) can be written as follows:
\begin{equation}\label{3.4}
X(t_{n}+h_n)=e^{h_n\mathcal{S}_n}(X(t_n))+h_n\int_{0}^{1}e^{(1-s)h_n\mathcal{S}_n}(\mathcal{G}_n(X(t_n+sh_n)))ds.
\end{equation}
The above expression has a similar structure with (\ref{1.5}) but Sylvester operator exponential instead of matrix exponential.
Thus we can apply various existing exponential integrators to (\ref{3.1}). The application of the general exponential Runge-Kutta type methods \cite{MH2}, to the MRDEs (\ref{3.1}) yields
\begin{equation}\label{3.5}
\left\{
\begin{array}{l}
X_{ni}=e^{c_ih_n\mathcal{S}_n}(X_n)+h_n\sum\limits^{i-1}_{j=1}a_{ij}(h_n\mathcal{S}_n)(\mathcal{G}_n(X_{nj})),~~1\leq i\leq s,\\
X_{n+1}=e^{h_n\mathcal{S}_n}(X_n)+h_n\sum\limits^{s}_{i=1}b_{i}(h_n\mathcal{S}_n)(\mathcal{G}_n(X_{ni})).
\end{array}
\right.
\end{equation}
Here, $c_i$ is the nodes, and the coefficients $a_{ij}(z),$ $b_i(z)$ are linear combinations of the $\varphi_j(c_iz),$ $\varphi_j(z)$ respectively.
These coefficients can be determined by a stiff error analysis which can be adapted from the stiff order theory presented in \cite{Hochbruck2010,Luan2013}. The process
is highly sophisticated and is omitted here. In our context, we only consider two specific exponential integration schemes. They will be used in our numerical experiments in Section 5. The first and simplest exponential integration scheme is the exponential Rosenbrock-type Euler scheme
\begin{eqnarray}
\begin{array}{ll}
X_{n+1}&=e^{h_n\mathcal{S}_n}(X_n)+h_n\varphi_1(h_n\mathcal{S}_n)(\mathcal{G}_n(X_n))\\
&=X_n+h_n\varphi_1(h_n\mathcal{S}_n)(F(X_n)).
\end{array}\label{3.6}
\end{eqnarray}
The scheme is computationally attractive since it is second order with only one $\varphi$-function. The second scheme is order three (denoted Erow3), which
can be regarded as a modification of exponential Rosenbrock-type Euler scheme
\begin{equation}
\begin{array}{lll}
X_{n,2}&=e^{h_n\mathcal{S}_n}(X_n)+h_n\varphi_1(h_n\mathcal{S}_n)(\mathcal{G}_n(X_n))\\
&=X_n+h_n\varphi_1(h_n\mathcal{S}_n)(F(X_n)),\\
X_{n+1}&=e^{h_n\mathcal{S}_n}(X_n)+h_n(\varphi_1(h_n\mathcal{S}_n)-2\varphi_3(h_n\mathcal{S}_n))(\mathcal{G}_n(X_{n}))+2h_n\varphi_3(h_n\mathcal{S}_n)(\mathcal{G}_n(X_{n,2}))\\
&=X_n+h_n\varphi_1(h_n\mathcal{S}_n)(F(X_n))+2h_n\varphi_3(h_n\mathcal{S}_n)(\mathcal{G}_n(X_{n,2})-\mathcal{G}_n(X_{n})).
\end{array}\label{3.7}
\end{equation}
The internal stage has the same
structure as the exponential Rosenbrock-type Euler scheme (\ref{3.6}), and the external stage is a perturbation of the internal stage. The above two schemes are usually embedded to create an adaptive time stepping method.

\section{Implementation issues }
For exponential integrators, the main computational cost is to approximate the exponential and exponential-type functions at each time-step. To our knowledge, there is no explicit method to evaluate the functions of a Sylvester operator in the literatures. For the computation of the first $\varphi$-function, the following formula gives an indirect way.

Define the augmented matrix $\mathbb{A}_n$ by
\begin{equation}
\begin{aligned}\label{4.5}
\mathbb{A}_n=\left(\begin{tabular}{cccccc}
$A_n$ & $\mathcal{G}_n$\\
$0$ & $-D_n$
\end{tabular}%
\right)\in \mathbb{C}^{(M+N)\times (M+N)}.
\end{aligned}
\end{equation}
Using the formula (10.40) arising in (\cite{Higham 2008}, we have
\begin{equation}
\begin{aligned}\label{4.7}
e^{\mathbb{A}_n}=\left(\begin{tabular}{cccccc}
$e^{A_n}$ & $\int_{0}^{1}e^{(1-s)A_n}\mathcal{G}_ne^{-sD_n}ds$\\
$0$ & $e^{-D_n}$
\end{tabular}%
\right)\in \mathbb{C}^{(M+N)\times (M+N)}.
\end{aligned}
\end{equation}
Then the scheme (\ref{3.6}) can be rewritten as
\begin{equation}\label{4.8}
\begin{aligned}
X_{n+1}=\left((I_M,0)e^{h_n\mathbb{A}_n}\left(\begin{tabular}{cccccc}
$X_n$\\
$I_M$
\end{tabular}%
\right)\right)e^{h_nD_n}.
\end{aligned}
\end{equation}

For the computation of a single matrix exponential or its action on a thin matrix, a number of methods have been proposed in the literatures for carrying out this task, see e.g. \cite{Al-Mohy 2009,Al-Mohy 2011,Sidje 1998} and the review \cite{Moler 2003}. This approach has the major advantage of simplicity but is likely to be too expensive for large $M$ and $N.$

A more general strategy for approximating $\varphi$-functions is to apply a numerical quadrature scheme. For a given function $\varphi_k$, the Sylvester operator $\mathcal{S}_n$ and an matrix $N_k,$ we approximate $\varphi_k(h_n\mathcal{S}_n)(N_k)$ by a quadrature approximation with quadrature nodes $s_j$ and weights $\omega_j:$
\begin{eqnarray}\label{4.10}
\varphi_k(h_n\mathcal{S}_n)N_k\approx \frac{1}{(k-1)!}\sum\limits^{p}_{j=0}\omega_j s_j^{k-1}e^{(1-s_j)h_n\mathcal{S}_n}(N_k).
\end{eqnarray}
Thus to evaluate $\varphi_k(h_n\mathcal{S}_n)N_k$ we need to compute $p+1$ operator exponential acting on the same matrix.
In practical application if the matrix $N_k$ has a low-rank factorization $N_k=L_kD_kU_k^T$ where both $L_k$ and $U_k$ are full column rank
and $D_k$ is nonsingular, the block Krylov subspace method can be applied to the computation of operator exponential involved.

In fact, as shown in the schemes (\ref{3.6}) and (\ref{3.7}), every stage in an exponential integrator can be expressed as a linear combination
of the form
\begin{equation}\label{4.1}
\varphi_0(\hat{\mathcal{S}}_n)N_0+\varphi_1(\hat{\mathcal{S}}_n)N_1+\varphi_2(\hat{\mathcal{S}}_n)N_2+\cdots+\varphi_k(\hat{\mathcal{S}}_n)N_k,
\end{equation}
here $\hat{\mathcal{S}}_n=h_n\mathcal{S}_n,$ $N_i\in R^{M\times N}, i=0,1,\cdots,k.$
Using the recurrence relation (\ref{2.11b}) we can calculate (\ref{4.1}) recursively. Two alternatives are available. The first one is a forward recursion, i.e.,
\begin{equation}\label{4.2}
\left\{
\begin{array}{l}
W_0=N_0,\\
W_j=\hat{\mathcal{S}}_n(W_{j-1})+N_j,~~j=1,\cdots,k,
\end{array}
\right.
\end{equation}
then
\begin{equation}\label{4.3}
\sum\limits^{k}_{j=0}\varphi_j(\hat{\mathcal{S}}_n)N_j=\varphi_k(\hat{\mathcal{S}}_n)(W_k)+\sum\limits^{k-1}_{j=0}\frac{1}{j!}W_j.
\end{equation}
The main computational cost of this process includes the action of $k$ Sylvester operator and a $\varphi_k(\hat{\mathcal{S}}_n)(W_k)$.
Another approach is a backward recursion
\begin{equation}\label{4.4}
\left\{
\begin{array}{l}
W_k=\hat{\mathcal{S}}_n^{-1}N_k,\\
W_{j}=\hat{\mathcal{S}}_n^{-1}(N_{j}+W_{j+1}),~~j=k-1,\cdots,1.
\end{array}
\right.
\end{equation}

\begin{equation}\label{4.5a}
\sum\limits^{k}_{j=0}\varphi_j(\hat{\mathcal{S}}_n)N_j=\varphi_0(\hat{\mathcal{S}}_n)(N_0+W_1)-\sum\limits^{k}_{j=1}\frac{1}{(j-1)!}W_j.
\end{equation}
This process requires the computation of $k$ algebra Sylvester equations and a Sylvester operator exponential acting on matrix.

In many practical applications the MRDEs have an low-rank structure and the solution has the low-rank property. In such cases it is necessary to avoid forming the matrices $X_n$ explicitly, because this in general leads to dense computations.
In the remainder of the section we briefly introduce how to implement the above mentioned two schemes in a low-rank fashion to the symmetric MRDEs. For simplicity let us consider the symmetric MRDEs (\ref{1.2}). 

Provided $Q,$ $G$ and $X_0$ in (\ref{1.2}) are symmetric positive semi-definite, and are given in the low-rank form
\begin{equation}\label{4.5b}
Q=C^TC, G=BB^T, \text{and}~X_0=L_0D_0L_0^T
\end{equation}
with $C\in \mathbb{R}^{l\times N},$ $B\in \mathbb{R}^{ N\times q},$ $L_0\in \mathbb{R}^{ N\times r},$ and $D_0\in \mathbb{R}^{ r\times r},$  $l,q,r\ll N.$ This implies that the solution $X(t)$ to the MRDEs (\ref{1.2}) is also symmetric positive semi-definite for all $t>0.$ First, we consider the exponential Rosenbrock-type Euler scheme (\ref{3.6}). Assume that the previous solution approximations $X_n$ admit a decomposition of the form $X_n = L_nD_nL_n^T$ with $L_n\in R^{N\times r_n}, D_n\in R^{r_n\times r_n}.$ $F_n$ in scheme (\ref{3.6}) can be written the form of $LDL^T:$
\begin{eqnarray}
\begin{array}{llll}\label{4.20}
F_n &=C^TC+AX_n+X_nA^T-X_nBB^TX_n\\
&=[C^T~ AL_n~ L_n~ L_n]\left(\begin{tabular}{cccccc}
$I$ & $$& $$& $$\\
$$ & $$& $D_n$& $$\\
$$ & $D_n$& $$& $$\\
$$ & $$& $$& $-(D_nL_n^TB)(D_nL_n^TB)^T$
\end{tabular}%
\right)[C^T~ AL_n~ L_n~ L_n]^T\\
&=[C^T~ AL_n~L_n]\left(\begin{tabular}{cccccc}
$I$ & $$& $$\\
$$ & $$& $D_n$\\
$$ & $D_n$& $-(D_nL_n^TB)(D_nL_n^TB)^T$\\
\end{tabular}%
\right)[C^T~ AL_n~ L_n]^T\\
&=\tilde{L}_n\tilde{D}_n\tilde{L}_n^T.\\
\end{array}
\end{eqnarray}
The new matrix $\tilde{L}_n$ has more columns than $L_n,$ and also more than their rank. As the number of columns in the decomposition increases, the computation cost will become prohibitively expensive. To overcome this difficulty we can incorporate the column compression strategy \cite{Lang} to $\tilde{L}_n,~\tilde{D}_n,$ and find more suitable low-rank factors.
Then, we apply the numerical quadrature formula (\ref{4.10}) to approximate $h_n\varphi_1(h_n\mathcal{S}_n)F_n,$ and the decomposition $Y_nT_nY_n^T$ is given by the factors
\begin{eqnarray}\label{4.19b}
Y_n=[e^{(1-c_0)h_nA_n}\tilde{L}_n,e^{(1-c_1)h_nA_n}\tilde{L}_n,\cdots,e^{(1-c_p)h_nA_n}\tilde{L}_n],
\end{eqnarray}
and
\begin{eqnarray}\label{4.19c}
T_n=diag(\gamma_0\tilde{D}_n,\gamma_1\tilde{D}_n,\cdots,\gamma_p\tilde{D}_p),~~\gamma_j=h_n\omega_j, j=1,\cdots,p.
\end{eqnarray}
Note that the evaluation of $Y_n$ requires computation of $p+1$ products between matrix exponential and a thin matrix. For large matrix $A_n,$ the block Krylov projection algorithm is a popular choice \cite{Saad}. An advantage of this computation is that one can project the $p+1$ operator exponential in the same search subspace $\mathcal{K}_{m}(A_n,\tilde{L}_n)$ and evaluate them simultaneously.

Now, using the splitting of $h_n\varphi_1(h_n\mathcal{S}_n)F_n$ and of the solution $X_n = L_nD_nL_n^T,$ the approximation $L_{n+1}D_{n+1}L_{n+1}^T$ to $X_{n+1}$ is given by
\begin{eqnarray}
\begin{array}{llll}
L_{n+1} =[L_n,Y_n],~~D_{n+1} =diag(D_n,T_n).
\end{array}
\end{eqnarray}
Again, we can employ column compression strategy to obtain low-rank splitting factors.

We now describe an alternative low-rank implementation of the exponential Rosenbrock-type Euler scheme. Apply the backward recursion (\ref{4.4})-(\ref{4.5a}) to the scheme (\ref{3.6}), giving
\begin{equation}\label{4.9}
\left\{
\begin{array}{l}
W_1=\mathcal{S}_n^{-1}F_n,\\
X(t_n+h_n)=\varphi_0(h_n\mathcal{S}_n)(W_1)+X_n-W_1, ~~n=0,1,\cdots.
\end{array}
\right.
\end{equation}
This results in solving one algebraic Lyapunov equation (ALE) which the right hand side has low-rank form in each time-step. There are many methods for solving Lyapunov equations where the right-hand side is of low rank, for instance by a low-rank ADI iteration \cite{Benner1} or Krylov subspace based methods \cite{Simoncini}. Due to the availability of low-rank ADI iteration based codes, here we limit ourselves to this procedure.

For order-third exponential integration scheme (\ref{3.7}), again, the previous solution approximation $X_n = L_nD_nL_n^T$ is assumed to be given in low-rank format. Note that the interval stage $X_{n,2}$ is the same as exponential Rosenbrock-type Euler scheme, thus $X_{n,2}$ can be written in the low-rank form $L_{n,2}D_{n,2}L_{n,2}^T.$ The external stage $X_{n+1}$ is a perturbation of the matrix $X_{n,2}$
by $h_n\varphi_3(h_n\mathcal{S}_n)(\mathcal{G}_n(X_{n,2})-\mathcal{G}_n(X_{n})).$ In order to find a low-rank factorization of the entire right hand side, we first consider the $LDL^T$-type splitting of $\mathcal{G}_n(X_{n,2})-\mathcal{G}_n(X_{n}).$ Direct calculation shows that
\begin{eqnarray}\label{4.21a}
\mathcal{G}_n(X_{n,2})-\mathcal{G}_n(X_{n})=X_nBB^TX_{n,2}+X_{n,2}BB^TX_n-X_{n,2}BB^TX_n-X_nBB^TX_n.
\end{eqnarray}
Inserting the splitting factors of $X_{n}$ and $X_{n,2}$ into (\ref{4.21a}) finally gives the splitting $\bar{L}_n\bar{D}_n\bar{L}_n^T$ with
\begin{eqnarray}\label{4.22}
\begin{array}{ll}
\bar{L}_n=[L_n~ L_{n,2}],~~~
\bar{D}_{n}=\left(\begin{tabular}{cccccc}
$-(D_nL_n^TB)(D_nL_n^TB)^T$ & $(D_nL_n^TB)(D_{n,2}L_{n,2}^TB)^T$\\
$(D_{n,2}L_{n,2}^TB)(D_nL_n^TB)^T$ & $-(D_{n,2}L_{n,2}^TB)(D_{n,2}L_{n,2}^TB)^T$\\
\end{tabular}%
\right).\\
\end{array}
\end{eqnarray}
Again, using (\ref{4.10}), the splitting factors $\bar{Y}_n, \bar{T}_n$ of $h_n\varphi_3(h_n\mathcal{S}_n)(\mathcal{G}_n(X_{n,2})-\mathcal{G}_n(X_{n}))$ can be computed as follows:
\begin{eqnarray}
\bar{Y}_n =[e^{(1-c_0)h_nA}\bar{L}_n,\cdots,e^{(1-c_p)h_nA}\bar{L}_n]
\end{eqnarray}
and
\begin{eqnarray}
\bar{T}_n =diag(\frac{\omega_0 c_0^{2}}{2}\bar{D}_n,\cdots,\frac{\omega_p c_p^{2}}{2}\bar{D}_n).
\end{eqnarray}
Now, using the $LDL^T$-type splitting with $X_{n,2}=L_{n,2}D_{n,2}L_{n,2}^T$ we obtain
\begin{eqnarray}
\begin{array}{lll}\label{3.9}
X_{n+1}\approx L_{n+1}D_{n+1}L_{n+1}^T\\
\end{array}
\end{eqnarray}
with
\begin{eqnarray}
\begin{array}{llll}
L_{n+1} =[L_{n,2},\bar{Y}_n]
\end{array}
\end{eqnarray}
and
\begin{eqnarray}
D_{n+1} =diag(D_{n,2},\bar{T}_n).
\end{eqnarray}

In actual implementation, once the new splitting factors are formed, column compression strategy should be performed to eliminate the redundant information.

\section{Numerical experiments}
In this section, we present some numerical experiments to illustrate the behaviour of exponential integration methods. We compare the numerical performance of the exponential Rosenbrock-type Euler scheme (\ref{3.6}) (denoted \textsf{ExpEuler}) and the third order exponential integration scheme \textsf{Erow3} (\ref{3.7}) with the splitting schemes in \cite{Stillfjord2}.
 For \textsf{ExpEuler}, we consider all the above mentioned three different implementations. They are marked as follows: the general implementation (\ref{4.8}) (denoted by \textsf{GExpEuler}), the low-rank implementation (denoted by \textsf{LrExpEuler}) and the backward recursion implementation (\ref{4.9}) (denoted by \textsf{BrExpEuler}). For low-rank implementations, the tolerance for column compression strategies are set to $n\cdot \epsilon,$ where $n$ is the system dimension and $\epsilon$ denotes the machine precision.
All experiments are performed under Windows 10 and MATLAB R2018b running on a laptop with an Intel Core i7 processor with 1.8 GHz and RAM 8 GB. Unless otherwise stated, we use the relative errors at the final time, measured in the Frobenius norm.

\noindent\textbf{Experiment 1.} As the first test, we consider the matrix $A\in \mathbb{R}^{n\times n}$ obtained from the standard 5-point difference discretization of the two-dimensional PDE
\begin{equation}\label{5.1}
\Delta u - f_1(x,y) \frac{\partial u}{\partial x} - f_2(x,y)\frac{\partial u}{\partial y}-f_3(x,y)u =0
\end{equation}
on the domain $\Omega=[0,1]^2$ with homogeneous Dirichlet boundary conditions, and $B,C\in \mathbb{R}^{n\times 2}$ with entries chosen randomly from [0; 1]. The matrix $A$ is a negative stiffness matrix which can be generated by MATLAB function \textsf{fdm2dmatrix} from LYAPACK toolbox \cite{Penzl}. We consider the discretization of (\ref{5.1}) for two different values of functions $f_1,f_2,f_3,$ namely $f_1=f_2=f_3=0$ and $f_1=10x,$ $f_2=100y,$ $f_3=0,$ respectively. The former generate a symmetric matrix (denoted \textsf{fdm-sym}), while the latter is unsymmetric (denoted \textsf{fdm-nonsym}).
The corresponding initial values $X_0$ are choosen as a low-rank product $X_0=L_0L_0',$ where $L_0\in \mathbb{R}^{n\times 2}$ are randomly generated.
To ensure the availability of a reference solution, we perform two sizes on the time interval [0, 1], one with $n = 64$ and the other with $n = 100.$ The reference solutions are obtained by MATLAB built-in function \textsf{ode15s} with an absolute tolerance of $10^{-20}$ and a relative tolerance of $2.22045\cdot10^{-14}.$ This has been done by vectorizing the MRDEs into a vector-valued ODEs with $n^2$ unknowns. We use the above mentioned methods to integrate the four systems over the time interval [0, 1] with time step size $h=0.01.$

Table \ref{tab1} lists the relative errors at the final time $t=1$ as well as the total time (in seconds) of the methods to integrate these systems. The results show that the two exponential integration schemes achieve the high precision of about $10^{-14}$ in all cases and obtained a higher order of convergence than we expected. An interpretation of this as the exponential integrators could be suitable for the structure of the MRDEs and capture some qualitative properties. As a comparison, we also present the results for the additive symmetric scheme of order 4 (denoted \textsf{Additive4}) in \cite{Stillfjord2} with the same timestep and \textsf{ode15s} with the same accuracy. The code for the \textsf{Additive4} contains parallel loops which uses 4 workers on our machine. We can see the exponential integration schemes accomplish higher computational accuracy than \textsf{Additive4} and take less runtimes than \textsf{ode15s}.

Figures \ref{figure1},~\ref{figure2} show plots of the F-norm of solutions obtained by \textsf{ExpEuler} and the reference solutions provided by \textsf{ode15s} for each test system. The \textsf{Erow3} yields very similar behaviors with \textsf{ExpEuler} and is omitted here. From these figures we see that the \textsf{ExpEuler} follow well behaviours of the reference solutions. Although it is not very accurate in the start some time steps, the behaviours are completely corrected as the time increase. At the final time, the relative error even level out around $10^{-14}.$ This is also true in the subsequent experiments and we interpret this as exponential integration schemes being favorable for MRDEs.
\begin{table}[h]
\setlength{\abovecaptionskip}{0.cm}
\setlength{\belowcaptionskip}{-0.3cm}
\caption{CPU time in seconds and relative errors for each of the methods over the time interval [0, 1]}
 \label{tab1}
\begin{center}
{\scriptsize
\begin{tabular*}{\textwidth}{@{\extracolsep{\fill}}@{~}c|c|lr|lr|lr|lr|lr|lr}
\toprule%
\raisebox{-2.00ex}[0cm][0cm]{$matrix$}&
\raisebox{-2.00ex}[0cm][0cm]{size(A)}
 &\multicolumn{2}{c|}{\textsf{GExpEuler}}&\multicolumn{2}{c|}{\textsf{LrExpEuler}}&\multicolumn{2}{c|}{\textsf{BrExpEuler}}&\multicolumn{2}{c|}{\textsf{Erow3}}
 &\multicolumn{2}{c|}{\textsf{Additive4}}&\multicolumn{2}{c}{\textsf{ode15s}} \cr
\cmidrule(lr){3-14}
&&Error&Time&Error&Time &Error&Time &Error&Time &Error&Time&Error&Time\\
\midrule%
\multirow{2}{*}{fdm-sym}&64$\times$64&1.22e-14&0.95&1.31e-14&3.36&4.58e-14&1.56&1.30e-14&4.25&2.09e-04&2.67&1.78e-14&334.81\\
&100$\times$100&1.57e-14&1.81&1.73e-14&6.28&4.46e-13&1.82&1.77e-14&7.51&8.53e-04&3.07& 2.29e-14&3708.20\\
\cmidrule(lr){3-14}
\multirow{2}{*}{fdm-nonsym}&64$\times$64&2.01e-14&0.96&2.16e-14&2.12&8.61e-14&2.12&2.15e-14&3.11&5.17e-04&4.02&2.89e-14&423.85\\
  &100$\times$100&2.26e-14&2.16&2.78e-14&12.67&3.21e-14&2.58&2.79e-14&15.37&1.89e-03&5.73&3.19e-14&4773.10\\
\bottomrule
\end{tabular*}
}
\end{center}
\end{table}

\begin{figure}
\begin{minipage}[t]{0.5\linewidth}
\centering
\includegraphics[width=7cm,height=4.5cm]{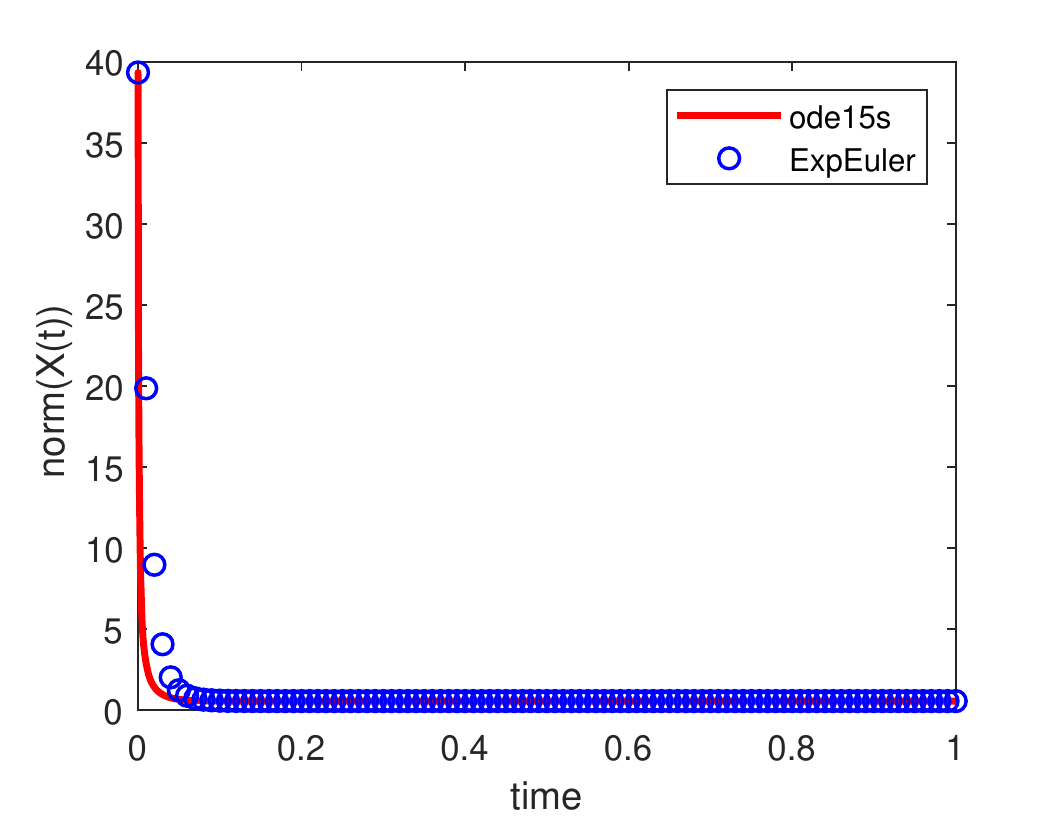}\\
\footnotesize{a. ~~\textsf{fdm-sym}, $n=64$}
\label{a}
\end{minipage}
\mbox{\hspace{-1.5cm}}
\begin{minipage}[t]{0.5\linewidth}
\centering
\includegraphics[width=7cm,height=4.5cm]{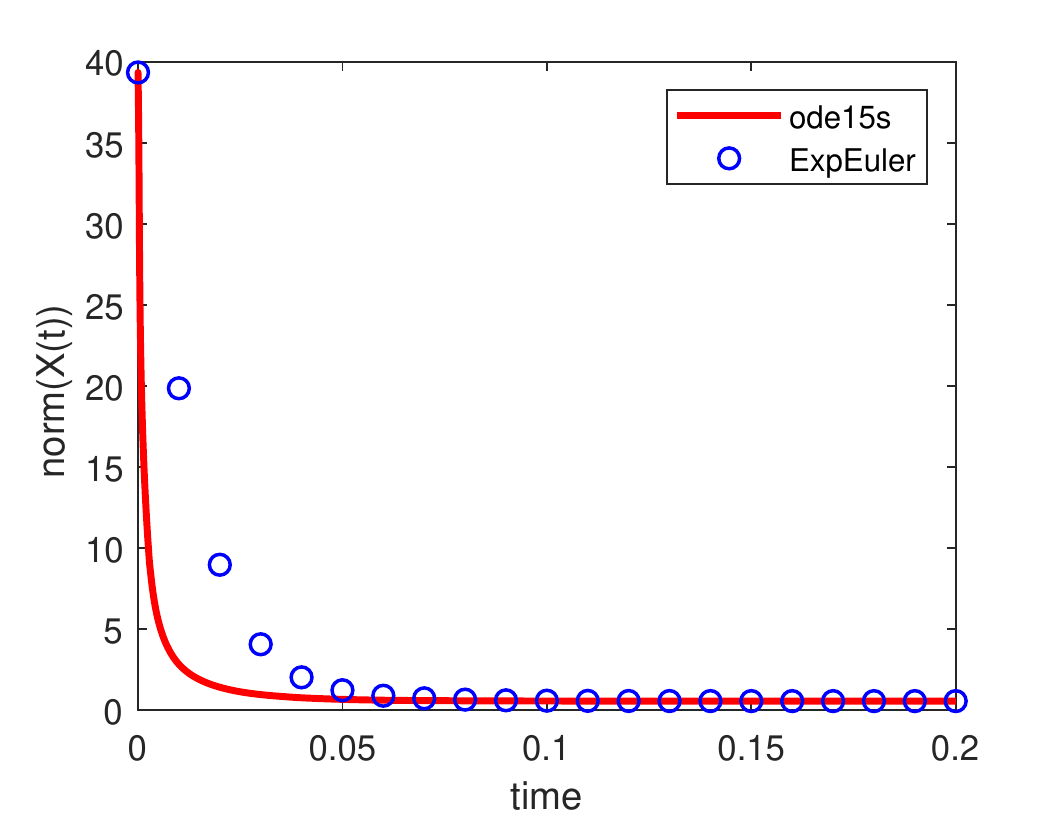}\\
\footnotesize{b.~~Zoomed image of a}
\label{EARRK2B }
\end{minipage}\\
\begin{minipage}[t]{0.5\linewidth}
\centering
\includegraphics[width=7cm,height=4.5cm]{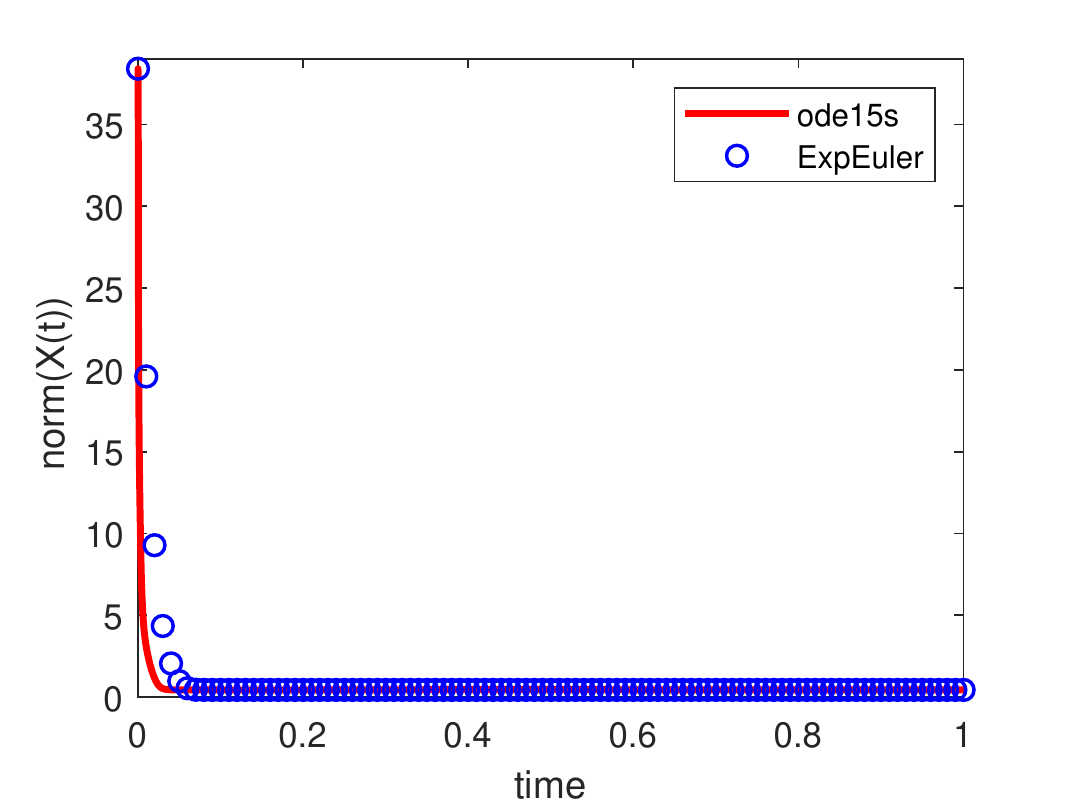}\\
\footnotesize{c.~~\textsf{fdm-nonsym}, $n=64$}
\label{EARRK2C}
\end{minipage}
\mbox{\hspace{-1.5cm}}
\begin{minipage}[t]{0.5\linewidth}
\centering
\includegraphics[width=7cm,height=4.5cm]{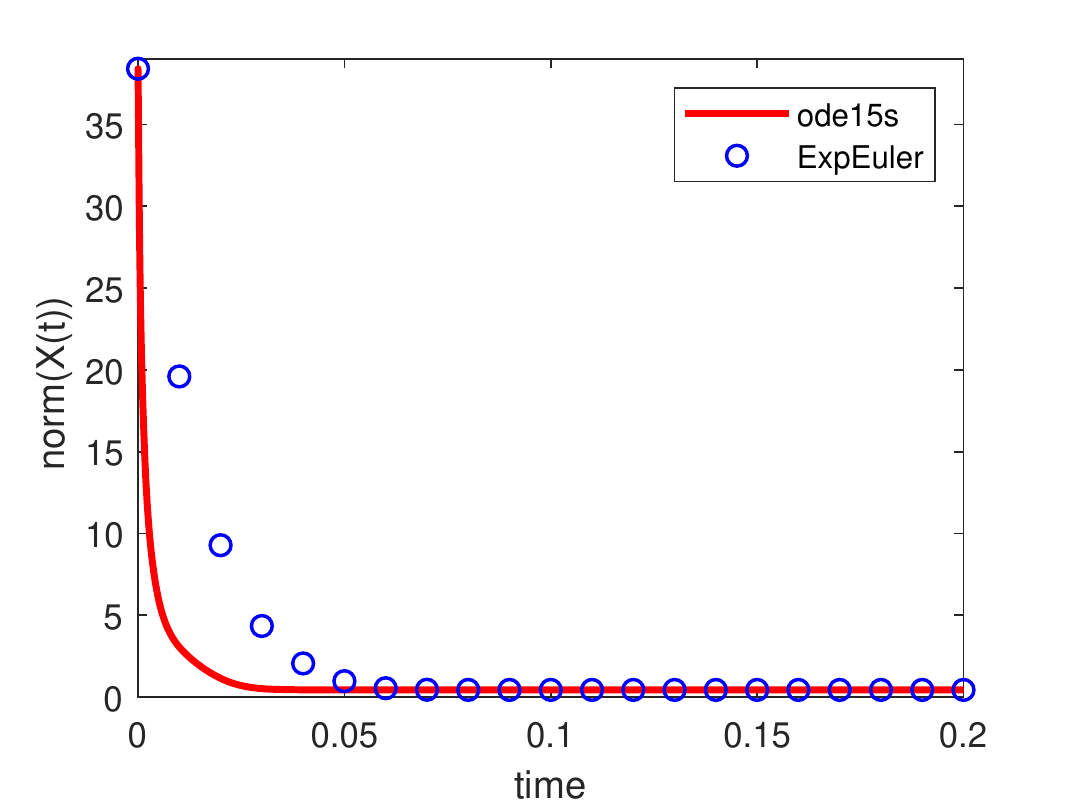}\\
\footnotesize{d.~~Zoomed image of c}
\label{EARRK2D}
\end{minipage}
\caption{ \footnotesize{The F-norm of the solutions using \textsf{ExpEuler} (o) and \textsf{ode15s} (-) for \textsf{fdm-sym} and \textsf{fdm-nonsym} of size $64\times64$ on [0,1], respectively.}}\label{figure1}
\end{figure}

\begin{figure}
\begin{minipage}[t]{0.5\linewidth}
\centering
\includegraphics[width=7cm,height=5cm]{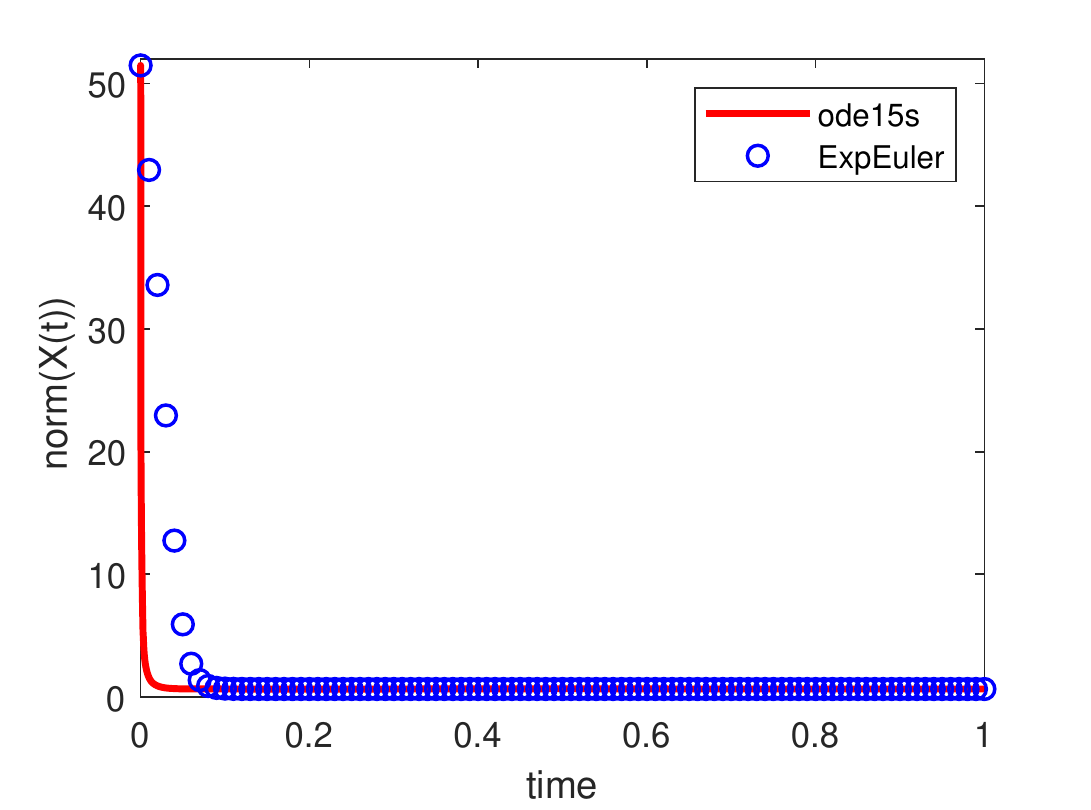}\\
\footnotesize{a. ~~\textsf{fdm-sym}, $n=100$}
\end{minipage}
\mbox{\hspace{-1.5cm}}
\begin{minipage}[t]{0.5\linewidth}
\centering
\includegraphics[width=7cm,height=5cm]{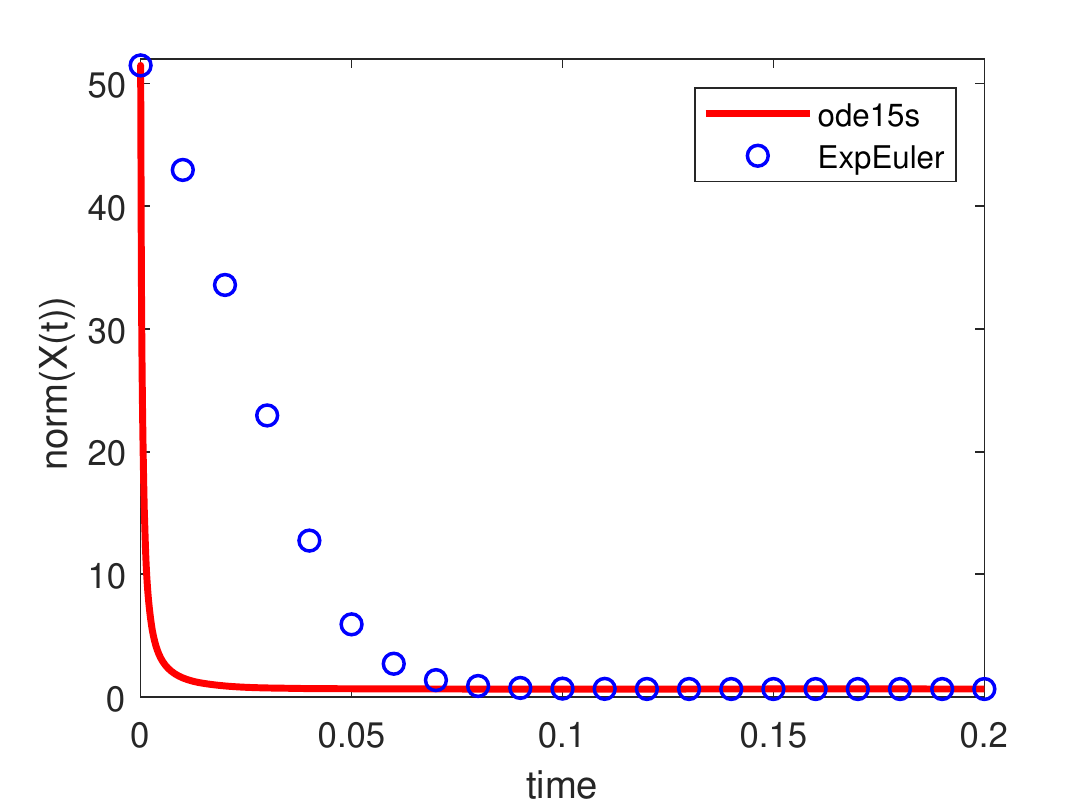}\\
\footnotesize{b.~~Zoomed image of a}
\end{minipage}\\
\begin{minipage}[t]{0.5\linewidth}
\centering
\includegraphics[width=7cm,height=5cm]{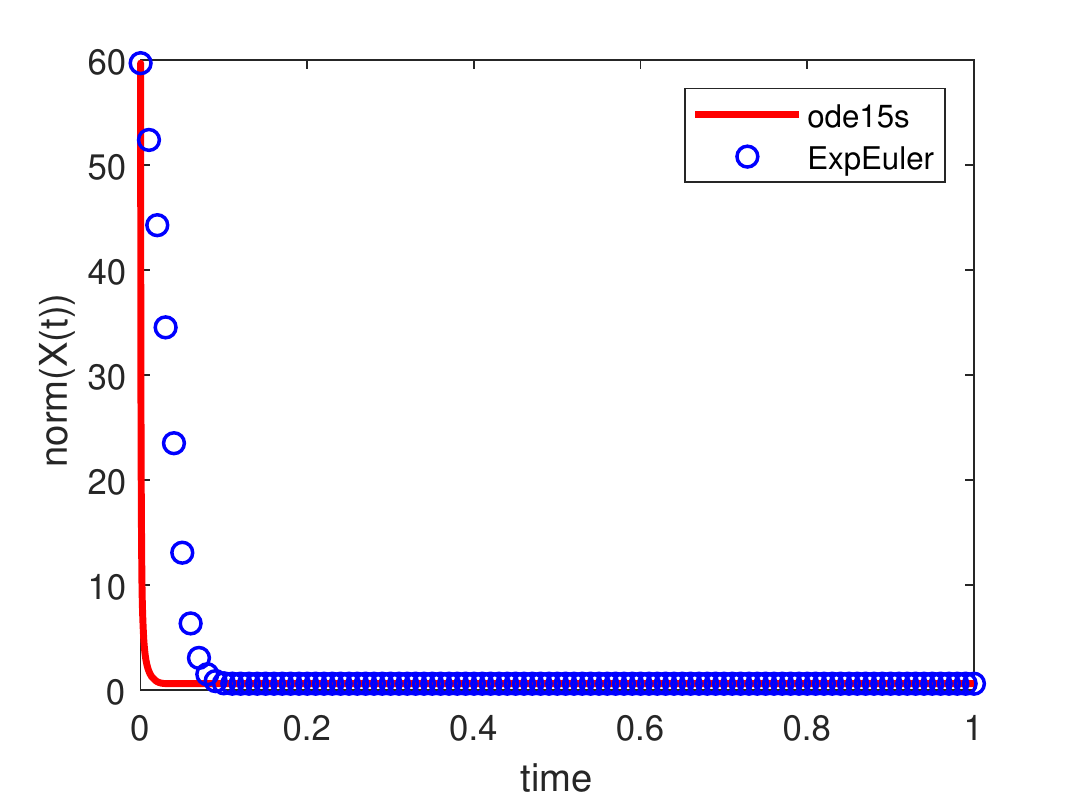}\\
\footnotesize{c.~~\textsf{fdm-nonsym}, $n=100$}
\end{minipage}
\mbox{\hspace{-1.5cm}}
\begin{minipage}[t]{0.5\linewidth}
\centering
\includegraphics[width=7cm,height=5cm]{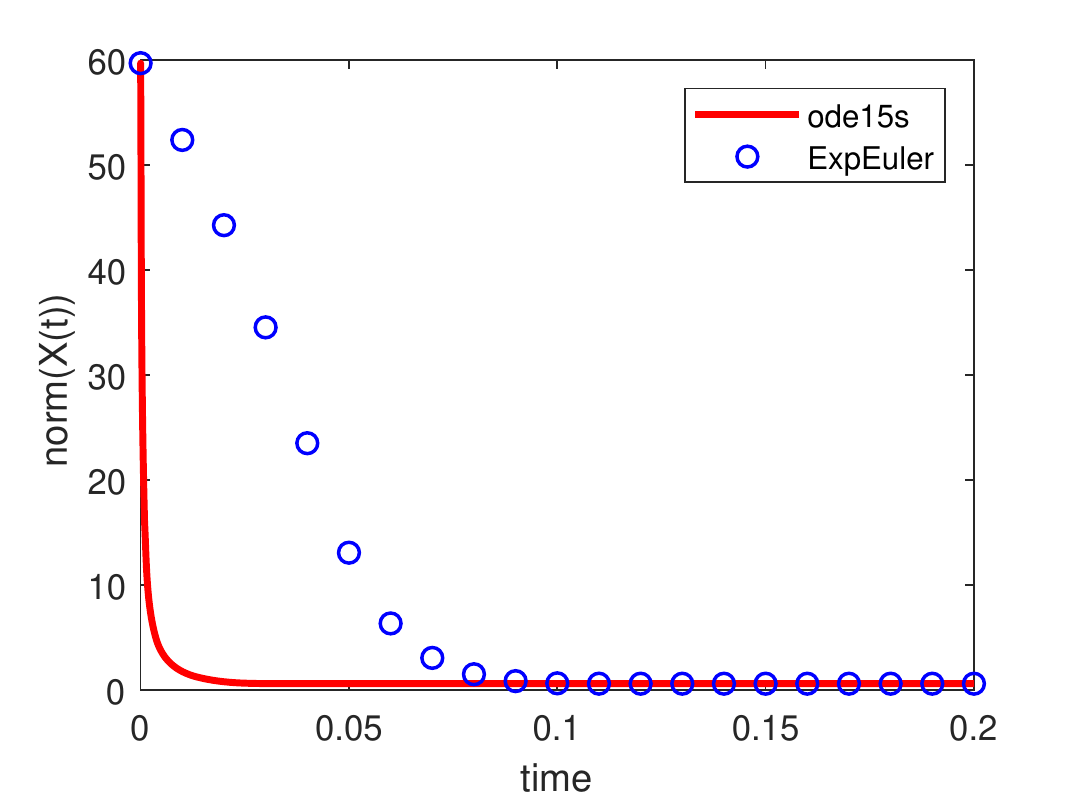}\\
\footnotesize{d.~~Zoomed image of c}
\end{minipage}
\caption{ \footnotesize{The F-norm of the solutions using \textsf{ExpEuler} (o) and \textsf{ode15s} (-) for \textsf{fdm-sym} and \textsf{fdm-nonsym} of size $100\times100$ on [0,1], respectively.}}\label{figure2}
\end{figure}

\noindent\textbf{Experiment 2.} As a second test, we consider a finite element discretization of a heat
equation arising from the optimal control of steel cooling \cite{Benner}. The matrix $A\in\mathbb{ R}^{n\times n}$ is symmetric and stable, $B\in \mathbb{R}^{n\times 7}$ and $C\in \mathbb{R}^{ 6\times n }.$ The initial value $X_0=L_0L_0'$ is set as $L_0=0_{n\times1}$ or random vectors $L_0=\textsf{rand}(n,2),$ with elements from the normal (0,1) distribution. To ensure the availability of a reference solution, we perform two sizes on the time interval [0,45] with $h=4.5$, one with $n = 371$ and the other with $n = 1357.$ The reference solutions are computed by MATLAB built-in function \textsf{ode45} with an absolute tolerance of $10^{-20}$ and a relative tolerance of $2.22045\cdot10^{-14}.$
Similar the above experiment Figures \ref{figure3} and \ref{figure4} show the convergence behaviour of the solutions of the \textsf{ExpEuler} to the reference solution
for the size $n = 371$ and $n=1375.$ From these Figures, we observe that the \textsf{ExpEuler} gives a fairly good
approximations of the reference solutions. We also present the numerical results of the different methods mentioned above in Table \ref{tab2}. We see that the proposed method performs quite well in terms of accomplished accuracy and computational time.

\begin{table}[h]
\setlength{\abovecaptionskip}{0.cm}
\setlength{\belowcaptionskip}{-0.3cm}
\caption{CPU time in seconds and relative errors for each of the methods over the time interval [0, 45] with $h=4.5$}
 \label{tab2}
\begin{center}
{\scriptsize
\begin{tabular*}{\textwidth}{@{\extracolsep{\fill}}@{~}c|c|lr|lr|lr|lr|lr}
\toprule%
\raisebox{-2.00ex}[0cm][0cm]{$matrix$}&\raisebox{-2.00ex}[0cm][0cm]{$L_0$} & \multicolumn{2}{c|}{\textsf{GExpEuler}}&\multicolumn{2}{c|}{\textsf{LrExpEuler}}&\multicolumn{2}{c|}{\textsf{BrExpEuler}}&\multicolumn{2}{c|}{\textsf{Erow3}}&
\multicolumn{2}{c}{\textsf{Additive4}} \\
\cmidrule(lr){3-12}
&&Error&Time &Error&Time &Error&Time&Error&Time &Error&Time\\
\midrule%
\multirow{2}{*}{rail371}
   &$0_{371\times1}$&8.92e-14&0.75&7.19e-12&0.11&7.07e-11&1.10&7.19e-12&0.16&3.38e-11&0.33\\
    &rand({371,2})&8.61e-14&0.99&4.23e-12&0.13&7.25e-11&1.18&4.22e-12&0.19&3.26e-11&0.42\\
\cmidrule(lr){3-12}
\multirow{2}{*}{rail1357}
   &$0_{1357\times1}$&2.23e-14&21.29&7.66e-12&0.56&2.48e-10&14.55&7.66e-12&0.80&3.60e-11&0.52\\
  &rand({1357,2}) &1.59e-14&29.54&4.18e-12&0.67&1.78e-10&16.11&4.27e-12&0.97&2.57e-11&0.56\\
\bottomrule
\end{tabular*}
}
\end{center}
\end{table}

\begin{figure}
\begin{minipage}[t]{0.5\linewidth}
\centering
\includegraphics[width=7cm,height=5cm]{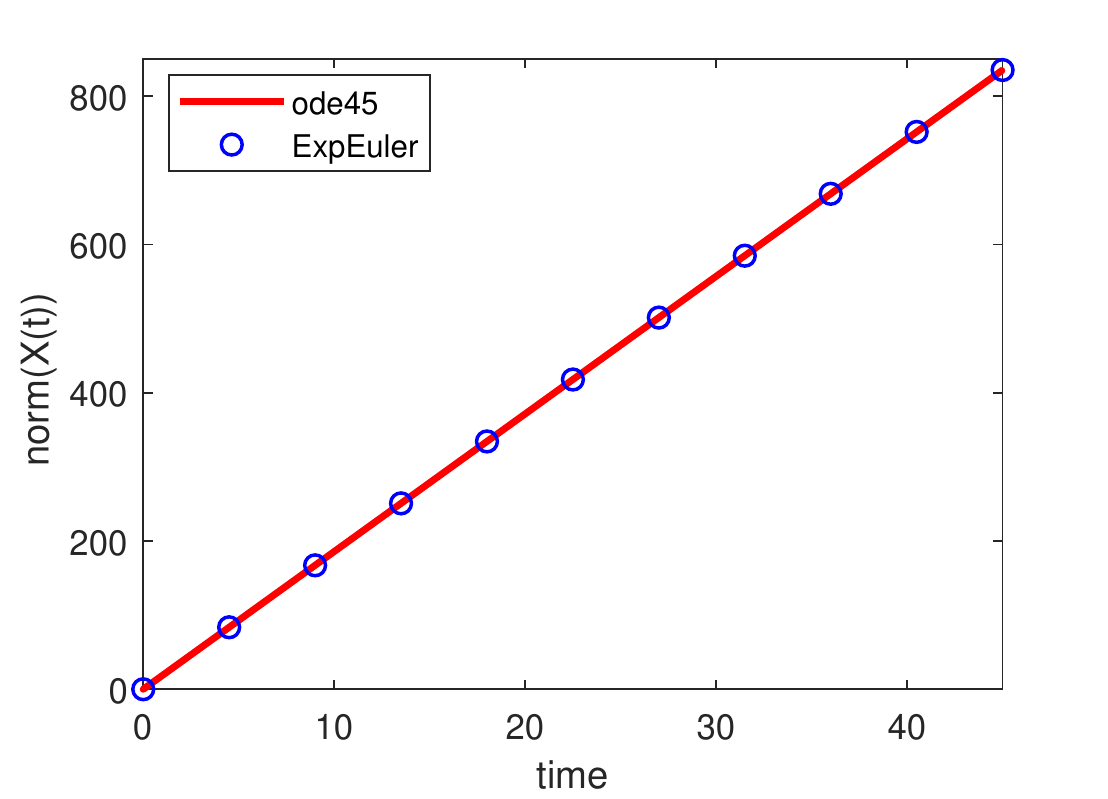}\\
\footnotesize{a. ~~$L_0=0_{371\times1}$}
\end{minipage}
\mbox{\hspace{-1.5cm}}
\begin{minipage}[t]{0.5\linewidth}
\centering
\includegraphics[width=6cm,height=5cm]{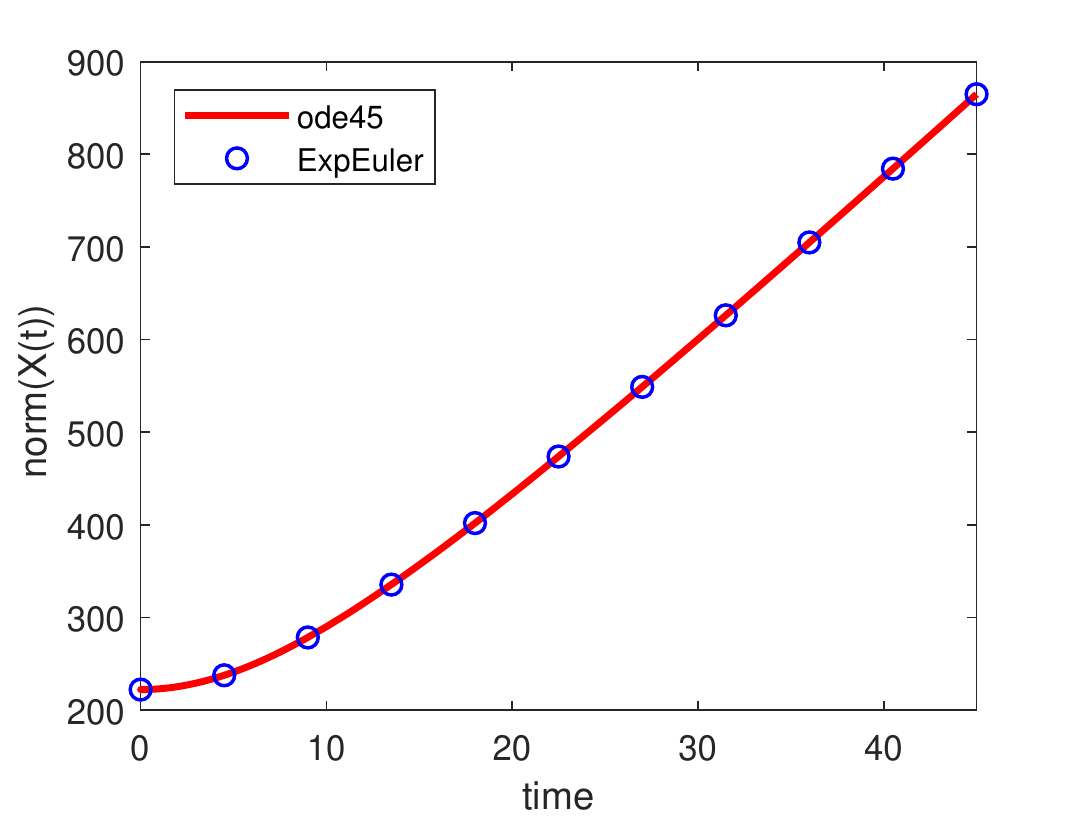}\\
\footnotesize{b.~~$L_0=rand_{371\times2}$}
\end{minipage}\\
\caption{\footnotesize{The F-norm of the solutions using \textsf{ExpEuler} (o) and \textsf{ode45} (-) for for $rail371$ with $X_0=L_0L_0'$ on [0,45] with $h=4.5.$}}\label{figure3}
\end{figure}

\begin{figure}
\begin{minipage}[t]{0.5\linewidth}
\centering
\includegraphics[width=7cm,height=5cm]{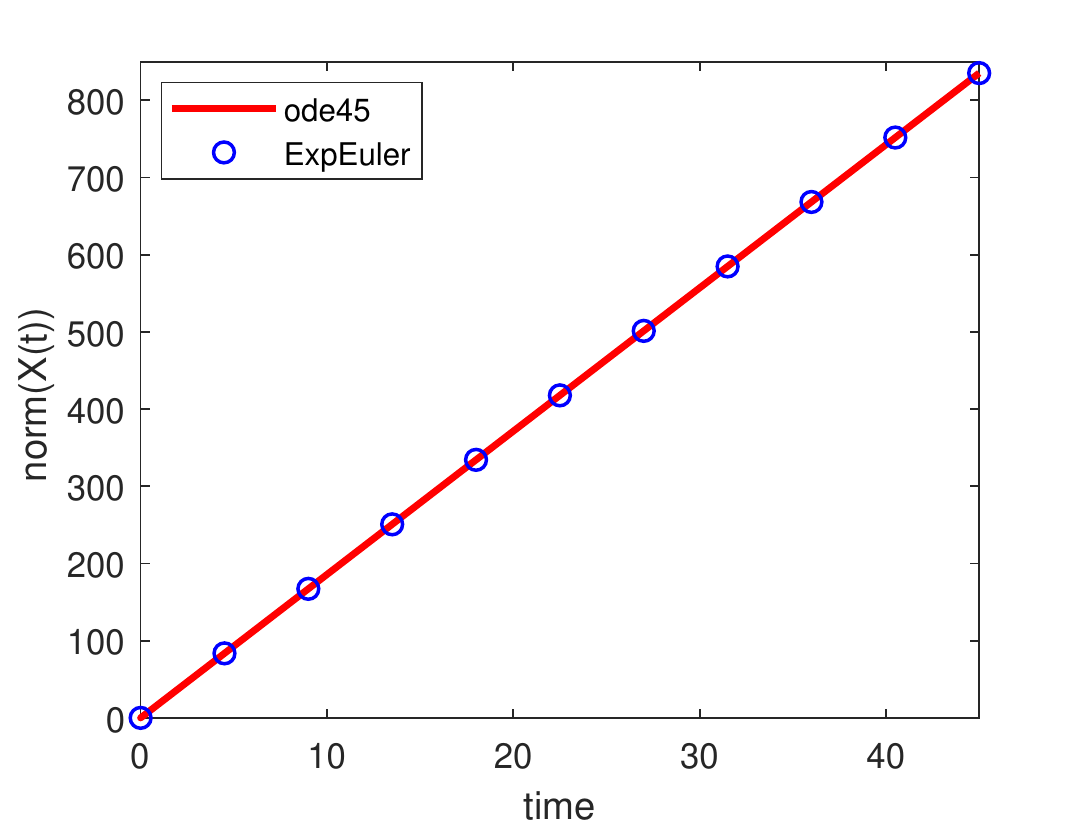}\\
\footnotesize{a. ~~$L_0=0_{1357\times1}$}
\end{minipage}
\mbox{\hspace{-1.5cm}}
\begin{minipage}[t]{0.5\linewidth}
\centering
\includegraphics[width=6cm,height=5cm]{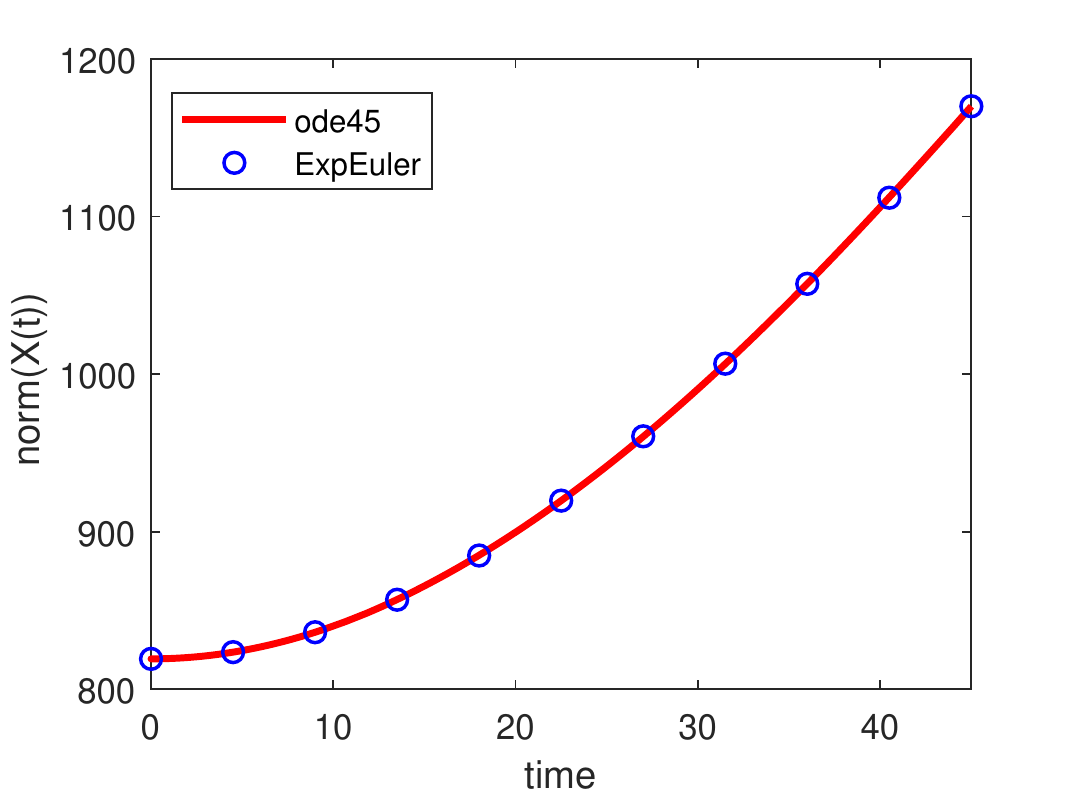}\\
\footnotesize{b.~~$L_0=rand_{1357\times2}$}
\end{minipage}\\
\caption{\footnotesize{The F-norm of the solutions using \textsf{LrExpEuler} (o) and \textsf{ode45} (-) for for $rail1357$ with $X_0=L_0L_0'$ on [0,45] with $h=4.5.$}}\label{figure4}
\end{figure}

\noindent\textbf{Experiment 3.} The third experiment is the same problem as in Experiment 1 for larger-scale dimensions. We choose the same setting as in Experiment 1. We compare the low-rank approximations, \textsf{LrExpEuler}, \textsf{Erow3}, \textsf{BrExpEuler} to the symmetric splitting of orders 4 and 6 for systems of dimensions $n=400,900,1600,2500.$ For larger-scale stiff matrix, our tests indicate that the numerical integration formula (\ref{4.10}) will cause low computational accuracy. In following tests, we use a relatively smaller sizes $h=0.001$ except for the \textsf{BrExpEuler} which is set as $h=0.01.$
Due to the systems sizes, it is infeasible to use \textsf{ode15s} to compute an accurate reference solutions. Instead, we use the eighth-order symmetric splitting scheme in \cite{Stillfjord2} with the time step $h=0.001.$ Table \ref{tab3} presents the relative errors and the corresponding computation times for each of the systems with different values of the system size $n.$
It is noted that the \textsf{BrExpEuler} produces the smallest errors of all methods even though the time step is larger. This is due to the poor performance  of numerical integrator to $\varphi$-function. This also shows that the exponential integrators have large computational potential and inspires us to exploit other efficiently numerical approximations to $\varphi$-function. We plan to investigate this option in our future work.
In terms of the CPU times,
we observe that in some cases the exponential integration methods spend more CPU times than the symmetric splitting schemes. This might have a weakened effect when the computation is done on a single core machine as the symmetric splitting methods employ parallel loops with four cores in our laptop. Again, the \textsf{LrExpEuler} obtained almost the same accuracy as \textsf{Erow3}, which illustrates the feasibility of adaptivity. We observe that the accuracy of the resulting solutions reduce as the size of system and its stiffness increase for all methods. We attribute this mainly to the pessimistic reference solutions.
\begin{table}[h]
\setlength{\abovecaptionskip}{0.cm}
\setlength{\belowcaptionskip}{-0.3cm}
\caption{The CPU and the relative errors with respect to the reference solution for $N=1000.$}
 \label{tab3}
\begin{center}
{\scriptsize
\begin{tabular*}{\textwidth}{@{\extracolsep{\fill}}@{~~}c|c|lr|lr|lr|lr|lr}
\toprule%
\raisebox{-2.00ex}[0cm][0cm]{Matrix} &\raisebox{-2.00ex}[0cm][0cm]{size(A)}&\multicolumn{2}{c|}{\textsf{LrExpEuler}}&\multicolumn{2}{c|}{\textsf{BrExpEuler}}&\multicolumn{2}{c|}{\textsf{Erow3}}&\multicolumn{2}{c|}{\textsf{Additive4}}&\multicolumn{2}{c}{\textsf{Additive6}} \\
\cmidrule(lr){3-12}
&&Error&Time &Error&Time &Error&Time &Error&Time&Error&Time\\
\midrule%
\multirow{4}{*}{fdm-sym} &400$\times$400&8.21e-07&16.64&1.46e-08&10.79&8.21e-07&41.31& 3.68e-05&64.69&8.03e-07&86.99\\
   &900$\times$900&7.67e-05&110.77&3.06e-06&77.50&7.67e-05&191.85&5.39e-04&148.54&4.18e-05&189.66\\
   &1600$\times$1600& 7.84e-04&500.91&6.21e-05&470.55&7.84e-04&711.84&2.53e-03&293.53&3.42e-04&348.13\\
   &2500$\times$2500&3.14e-03&1641.29&3.70e-04&1953.69&3.14e-03&2067.08&7.51e-03&541.69&1.15e-03&957.70\\
   \cmidrule(lr){3-12}
\multirow{4}{*}{fdm-nonsym} &400$\times$400& 1.18e-06&28.03&1.96e-08&19.37&1.18e-06&87.31&4.40e-05&116.79&1.03e-06&170.70\\
   &900$\times$900&7.80e-05&112.67&3.37e-06&116.65&7.80e-05&263.88&5.64e-04&392.62&4.51e-05&491.76\\
   &1600$\times$1600&8.16e-04&558.12&6.52e-05&702.36&8.16e-04&925.01&2.56e-03&917.33&3.58e-04&990.88\\
   &2500$\times$2500&3.22e-03&2013.00&3.92e-04&2822.60&3.22e-03&2629.20&8.00e-03&1671.76&1.22e-03&2120.35\\
\bottomrule
\end{tabular*}
}
\end{center}
\end{table}

\section{Conclusion}
In this paper, we show how to apply exponential integrators to get approximate solutions of large stiff MRDEs. The low-rank implementation of such schemes for large-scale applications and their comparison with current state-of-the-art integrators must be addressed. Numerical experiments illustrate that the exponential integration methods can achieve convergence than the expected order. Thus the exponential integrators can provide an efficient alternative to standard integrators for large-scale stiff problems. The study of the performance and application of the higher-order exponential integration schemes and their comparative performance with implicit schemes will be presented elsewhere. We also plan to develop new more efficient algorithms to approximate the exponential and related functions of a Sylvester operator acting on an matrix.

\end{document}